\documentclass{article}

\usepackage{amsthm}
\usepackage{amssymb}
\usepackage{amsmath}

\usepackage{a4wide}

\newtheorem{theorem}{Theorem}[section]
\newtheorem{lemma}[theorem]{Lemma}
\theoremstyle{remark}
\newtheorem{remark}{Remark}[section]
\theoremstyle{definition}
\newtheorem{definition}{Definition}[section]

\begin{document}

\title{Global Existence of Solutions for a Fractional Caputo 
Nonlocal Thermistor Problem\thanks{This is a preprint of a paper 
whose final and definite form is with 'Advances in Difference Equations', 
ISSN 1687-1847, available at [{\tt https://advancesindifferenceequations.springeropen.com}].
Submitted 21/July/2017; Revised 21/Oct/2017; Accepted for publication 31/Oct/2017.}}

\author{Moulay Rchid Sidi Ammi and Ismail Jamiai\\[0.15cm]
Department of Mathematics, AMNEA Group\\
Faculty of Sciences and Techniques\\
Moulay Ismail University\\
B.P. 509, Errachidia, Morocco\\
\texttt{sidiammi@ua.pt, i.jamiai@gmail.com} \\[0.3cm]
\and
Delfim F. M. Torres\\[0.15cm]
Department of Mathematics, R\&D Unit CIDMA\\
University of Aveiro, 3810-193 Aveiro, Portugal\\
\texttt{delfim@ua.pt}}

\date{}

\maketitle

% -------------------------------------------------

\begin{abstract}
We begin by proving a local existence result for a fractional 
Caputo nonlocal thermistor problem. Then, additional existence 
and continuation theorems are obtained, ensuring global existence 
of solutions.  
\end{abstract}

% -------------------------------------------------

\smallskip


\smallskip

% -------------------------------------------------

\smallskip


\medskip

% -------------------------------------------------

\section{Introduction}

Fractional calculus is acknowledged as an important research tool 
that opens up many horizons in the field of dynamical systems 
\cite{book:frac:ICP2}. According to Professor Katsuyuki Nishimoto, 
``the fractional calculus is the calculus of the XXIst century'' \cite{MR1101141}. 
This opinion is strengthen by a huge increase of interest in this research tool, 
expressed by an increase in the number of theoretical developments 
and basic theory on this subject; see, e.g., \cite{11,17,29,28,29bis}. 
Recently, it has also been proved that fractional 
differential equations are significant and essential tools when 
applied in the study of nonlocal or time-dependent processes 
and in the modeling of many applications, including chaotic dynamics, 
material sciences, mechanic of fractal and complex media, 
quantum mechanics, physical kinetics, chemistry, biology, 
economics and control theory \cite{MR2736622}. 
For instance, a fractional generalization 
of the Newtonian equation to describe the dynamics 
of complex phenomena, in both science and engineering, 
has been proposed in \cite{Ref:1:2};
a fractional Langevin equation, with applications 
in polymer layers, has been investigated in \cite{Ref:1:1}.
One can say that real-world problems require definitions of fractional 
derivatives for initial and boundary value problems
\cite{MR3561379,MR2768178}. Fractional mathematical models 
describing natural phenomena, like shallow water waves and ion acoustic 
waves in plasma and vibration of large membranes, 
as well as personal and interpersonal realities,
like smoking, romantic relationships and marriages,
can be found in \cite{Ref:2:2,Ref:2:4}
and \cite{Ref:2:1,Ref:2:3}, respectively.
Details on the geometric and physical interpretation of 
fractional differentiation can be found in \cite{28}.

Thermistor is a thermo-electric device constructed  
from a ceramic material whose electrical conductivity 
depends strongly on the temperature. This makes thermistor 
problems highly nonlinear \cite{MathCompSim}. 
They can be used as a switching device in many electronic circuits. 
A broad application spectrum of thermistor problems
in heating processes and current flow can be found in several areas 
of electronics and its related industries \cite{MyID:211}. Generally, 
there are two kinds of thermistors: the first have an electrical conductivity that
decreases with the increasing of temperature; the second have an 
electrical conductivity that increases with the increasing of temperature 
\cite{kwok,maclen}. Here we consider a prototype of electrical conductivity 
that depends strongly in both time and temperature. Our goal consists 
to prove global existence of solutions for a fractional Caputo nonlocal 
thermistor problem. The results are obtained via Schauder's fixed point theorem.
Precisely, we consider the following fractional order initial value problem: 
\begin{equation}
\label{1}
\begin{gathered}
_{C}D^{2 \alpha}_{0, t} u(t) = \frac{\lambda f(t, u(t))}{(
\int_{0}^{t}f(x, u(x))\, dx)^{2}} \, , \quad  t \in  (0, \infty)  \, , \\
u(t)|_{t=0}=u_0,
\end{gathered}   
\end{equation}
where $ _{C}D^{2 \alpha}_{0, t}$ is the fractional Caputo derivative 
operator of order $2 \alpha$ with $0 < \alpha < \frac{1}{2}$ a real parameter. 
Function $u$ denotes the temperature and $\lambda$ is a positive real.
We shall assume the following hypotheses:
\begin{itemize}
\item[$(H_1)$] $f: \mathbb{R}^{+} \times \mathbb{R}^{+}\rightarrow \mathbb{R}^{+}$ 
is a Lipschitz continuous function with Lipschitz constant $L_{f}$ with respect 
to the second variable such that $c_{1} \leq f(s,u) \leq c_{2}$
with $c_{1}$ and $c_{2}$ two positive constants;

\item[$(H_2)$] there exists a positive constant $M$ such that $f(s, u) \leq M s^{2}$;

\item[$(H_3)$]  $|f(s, u)- f(s, v) | \leq  s^{2} |u-v|$ or, in a more general manner, 
there exists a constant $\omega \geq 2$ such that $|f(s,u)- f(s,v) | \leq  s^{\omega} |u-v|$.
\end{itemize}

In the literature, questions evolving existence and uniqueness of solution 
for fractional differential equations (FDEs) have been intensely studied 
by many mathematicians \cite{16,17,22,29}. However, much of published papers 
have been concerned with existence-uniqueness of solutions for FDEs 
on a finite interval. Since continuation theorems for FDEs are not well developed, 
results about global existence-uniqueness of solution of FDEs on the half axis 
$[0,+\infty )$, by using directly the results from local existence, 
have only recently flourished \cite{2,li}. 

In contrast with our previous works \cite{sidiammi1,MyID:347,MyID:365} 
on fractional nonlocal thermistor problems, which were focused 
on local existence and numerical methods, here we are concerned 
with continuation theorems and global existence for the steady state 
fractional Caputo nonlocal thermistor problem.  
The paper is organized as follows. In Section~\ref{section2}, 
we collect some background material and necessary results from
fractional calculus. Then, we are concerned in Section~\ref{section3} 
with local existence on a finite interval for \eqref{1} (Theorem~\ref{thm3.1}). 
Section~\ref{section4} is devoted to (non)continuation (Theorem~\ref{thm4.1}) 
associated with problem \eqref{1}, which allows to generalize 
the main result of Section~\ref{section3}. Our proofs rely on Schauder's 
fixed point theorem and some extensions of the continuation theorems 
for ordinary differential equations (ODEs) to the fractional order case. 
One of the main difficulties lies in handling the nonlocal term 
$\frac{\lambda f(t, u(t))}{(\int_{0}^{t}f(x, u(x))\, dx)^{2}}$, 
representing a heat source and that depends continuously on time; 
another one in the fact that electrical conductivity depends 
on both time and temperature. Based on the results of Section~\ref{section4},
in Section~\ref{section5} we prove existence of a global solution for
\eqref{1}: see Theorems~\ref{thm5.2} and \ref{thm5.2b}. 
We end with Section~\ref{section6} of conclusions.

% -------------------------------------------------

\section{Preliminaries and basic results}
\label{section2}

In this section, we collect from the literature \cite{17,22,25,23,41,29}
some background material and basic results that will be used in the 
remainder of the paper.

Let $C[a,b]$ be the Banach space of all real valued continuous functions on $[a,b]$
endowed with the norm $\| x\| _{[a,b]}=\max_{t\in [a,b]}| x(t)|$. According to 
the Riemann--Liouville approach to fractional calculus, we introduce
the fractional integral of order $\alpha$, $\alpha >0$, as follows.

\begin{definition} 
\label{def2.1}
The Riemann--Liouville integral of a function $g$
with order $\alpha >0$ is defined by
\begin{equation*}
_{RL}D_{0,t}^{-\alpha }g(t)=\frac{1}{\Gamma (\alpha )}
\int_0^{t}(t-s)^{\alpha -1}g(s)ds,\quad t>0,
\end{equation*}
where $\Gamma$ is the Euler gamma function given by
$$
\Gamma (\alpha) = \displaystyle \int_{0}^{\infty} t^{\alpha - 1} e^{-t}dt, 
$$
$\alpha >0$.
\end{definition}

The natural next step, after the notion of fractional integral has been introduced, 
is to define the fractional derivative of order $\alpha$, $\alpha >0$.

\begin{definition} 
\label{def2.2} 
The Riemann--Liouville derivative of function $g$
with order $\alpha >0$ is defined by
\begin{equation*}
_{RL}D_{0,t}^{\alpha }g(t)
=\frac{1}{\Gamma (n-\alpha )}\frac{d^{n}}{dt^{n}}
\int_0^{t}(t-s)^{n-\alpha -1}g(s)ds,\ t>0,
\end{equation*}
where $n-1<\alpha <n\in \mathbb{Z}^{+}$.
\end{definition}

Note the remarkable fact that, in the Riemman--Liouville  
sense, the fractional derivative of the constant function 
is not zero. We now give an alternative and more restrictive 
definition of fractional derivative, first introduced by Caputo  
in the end of 1960's \cite{66,65} and then adopted 
by Caputo and Mainardi in \cite{67,24}. In Caputo sense,
the fractional derivative of a constant is zero.

\begin{definition} 
\label{def2.3}
The Caputo derivative of function $g(t)$ with order $\alpha >0$ is defined by
\begin{equation*}
_{C}D_{0,t}^{\alpha }g(t)
=\frac{1}{\Gamma (n-\alpha)}
\int_0^{t}(t-s)^{\alpha -1}g^{(n)}(s)ds,\ t>0,
\end{equation*}
where $n-1<\alpha <n\in \mathbb{Z}^{+}$.
\end{definition}

For proving our main results, we make use of the following auxiliary lemmas.

\begin{lemma}[See \cite{li}] 
\label{lem2.2}
Let $M$ be a subset of $C([0,T])$. Then $M$ is precompact
if and only if the following conditions  hold:
\begin{enumerate}
\item $\{u(t):u \in M\}$ is uniformly bounded,

\item $\{u(t):u \in M\}$ is equicontinuous on $[0,T]$.
\end{enumerate}
\end{lemma}

\begin{lemma}[Schauder fixed point theorem \cite{li}] 
\label{lem2.3 }
Let $U$ be a closed bounded convex subset of a Banach space $X$. If
$T:U\to U$ is completely continuous, then $T$ has a fixed point in $U$.
\end{lemma}

Finally we recall a generalization of Gronwall's lemma, 
which is essential for the proof of our Theorem~\ref{thm5.2b}.

\begin{lemma}[Generalized Gronwall inequality \cite{15,35}] 
\label{lem5.1}
Let $v:[0,b]\to [0,+\infty )$ be a real function and $w(\cdot)$
be a nonnegative, locally integrable function on $[0,b]$.
Suppose that there exist $a>0$ and $0<\alpha <1$ such that
\begin{equation*}
v(t)\leq w(t)+a\int_0^{t}\frac{v(s)}{(t-s)^{\alpha }}ds.
\end{equation*}
Then there exists a constant $k=k(\alpha )$ such that 
\begin{equation*}
v(t)\leq w(t)+ka\int_0^{t}\frac{w(s)}{(t-s)^{\alpha }}ds
\end{equation*}
for $t\in [0,b]$.
\end{lemma}

% -------------------------------------------------

\section{Local existence theorem}
\label{section3}

In this section, a local existence theorem of solutions for \eqref{1} 
is obtained by applying Schauder's fixed point theorem. In order 
to transform \eqref{1} into a fixed point problem, we give in the 
following lemma an equivalent integral form of \eqref{1}.

\begin{lemma} 
\label{lem2.1}
Suppose that $(H_1)$--$(H_3)$ holds. Then the
initial value problem \eqref{1} is equivalent to
\begin{equation}
\label{2}
u(t)=u_0+\frac{\lambda}{\Gamma (2\alpha )}\int_0^{t}(t-s)^{2\alpha-1}  
\frac{f(s, u(s))}{\left( \int_{0}^{t}f(x, u)\,  dx\right)^{2}}  ds.  
\end{equation}
\end{lemma}

\begin{proof}
It is a simple exercise to see that $u$ is a solution of the integral equation \eqref{2} 
if and only if it is also a solution of the IVP \eqref{1}.
\end{proof}

\begin{theorem} 
\label{thm3.1}
Suppose that conditions $(H_1)$--$(H_3)$ are verified. 
Then \eqref{1} has at least one
solution $u\in C[0,h]$ for some $T\geq h>0$.
\end{theorem}

\begin{proof} 
Let
\begin{equation*}
E=\left\{ u\in C[0,T]:\| u-u_0\|_{C[0,T]}
=\sup_{0 \leq t\leq T}| u-u_0| \leq b\right\},
\end{equation*}
where $b$ is a positive constant.
Further, put
\begin{equation*}
D_{h}=\left\{ u:u\in C[0,h],\ \| u-u_0\| _{C[0,h]} \leq b\right\},
\end{equation*}
where 
$$
h=\min \left\{ \left (b \left(\frac{ \lambda M }{ 
\Gamma(2 \alpha+1) c_{1}^{2}}\right)^{-1}\right)^{\frac{1}{2 \alpha}},T\right\}
$$ 
and $ 0< \alpha < \frac{1}{2}$. It is clear that $h \leq T$. Note also that 
$D_{h}$ is a nonempty, bounded, closed, and convex subset of $C[0,h]$. In order 
to apply Schauder's fixed point theorem, we define the following operator $A$:
\begin{equation}
\label{3}
(Au)(t)=u_0+\frac{\lambda}{\Gamma (2\alpha )}
\int_0^{t}(t-s)^{2\alpha-1}\frac{f(s, u(s))}{\left( 
\int_{0}^{t}f(x, u)\,  dx\right)^{2}}  ds,\quad t\in [0,h].  
\end{equation}
It is clear that all solutions of \eqref{1} are fixed points of \eqref{3}.
Then, by assumptions $(H_1)$ and $(H_2)$, we have for any $u\in C[0,h]$ that
\begin{align*}
\left|(Au)(t)-u_0\right|  
& \leq \frac{ \lambda }{\Gamma(2 \alpha )} \frac{1}{(c_{1}t)^{2}} 
\int_0^{t}(t-s)^{2\alpha-1} f(s, u(s)) ds \\
& \leq   \frac{ \lambda }{ \Gamma(2 \alpha )} 
\frac{M}{c_{1}^{2}} \int_0^{t}(t-s)^{2\alpha-1} ds \\
& \leq  \frac{ \lambda M }{2\alpha  \Gamma(2 \alpha )} 
\frac{1}{c_{1}^{2}}  h^{2 \alpha}\\
&= \frac{\lambda M }{\Gamma(2 \alpha +1)} \frac{1}{c_{1}^{2}}  h^{2 \alpha}\\ 
&\leq b.
\end{align*}
It yields that $AD_{h}\subset D_{h}$. 
Our next step, in order to prove Theorem~\ref{thm3.1},
is to show that the following lemma holds.
\begin{lemma}
\label{lemma:in1}
The operator $A$ is continuous.
\end{lemma}
\begin{proof}
Let $u_n,\, u\in D_{h}$ be such that $\| u_n-u\| _{C[0,h]}\to 0$ 
as $n\to +\infty$. One has
\begin{equation}
\label{eq4}
\begin{aligned}
&|A u_{n}(t) - Au(t)|\\
& \leq  \frac{\lambda}{\Gamma (2 \alpha )} 
\int_{0}^{t}(t-s)^{2 \alpha -1}  
\left| \frac{f(s, u_{n}(s))}{( \int_{0}^{t}f(x, u_{n})\, d x)^{2}}\, 
- \frac{f(s, u(s))}{( \int_{0}^{t}f(x, u)\, d x)^{2}}\, \right| d s\\
& \leq  \frac{\lambda}{\Gamma (2 \alpha )} \int_{0}^{t}(t-s)^{2 \alpha -1}   
\left|  \frac{1}{( \int_{0}^{t}f(x, u_{n})\, d x)^{2}} \left(f(s, u_{n}(s))-f(s, u(s))\right)\right.\\
& \quad \left.+ f(s, u(s)) \left( \frac{1}{\left( \int_{0}^{t}f(x, u_{n})\, d x\right)^{2}}
-  \frac{1}{\left(\int_{0}^{t}f(x, u)\,  d x \right)^{2}} \right)  \right| ds \\
& \leq \frac{\lambda}{\Gamma (2 \alpha )} \int_{0}^{t}(t-s)^{2 \alpha -1} 
\frac{1}{( \int_{0}^{t}f(x, u_{n})\, d x)^{2}} \left|f(s, u_{n}(s))-f(s, u(s))\right| d s\\
& \quad +  \frac{\lambda}{\Gamma (2 \alpha )} \int_{0}^{t}(t-s)^{2 \alpha -1} |f(s, u(s))| 
\left|  \frac{1}{( \int_{0}^{t}f(x, u_{n})\, d x)^{2}}  
- \frac{1}{( \int_{0}^{t}f(x, u)\, d x)^{2}} \right|\\
& \leq  I_{1} + I_{2}.
\end{aligned}
\end{equation}
We now focus on both right hand terms separately. 
By hypotheses $(H_2)$ and $(H_3)$, we have
\begin{align*}
I_{1} 
& \leq \frac{\lambda}{(c_{1}t)^{2}\Gamma (2 \alpha )}
\int_{0}^{t}(t-s)^{2 \alpha -1}  \left|f(s, u_{n}(s))-f(s, u(s))\right| d s \\
& \leq \frac{\lambda L_{f}}{c_{1}^{2}\Gamma (2 \alpha )}
\int_{0}^{t}(t-s)^{2 \alpha -1}  \left|u_{n}(s)-u(s)\right| d s \\
& \leq \frac{\lambda L_{f}}{c_{1}^{2}\Gamma (2 \alpha)}\| u_{n}-u \|_{C[0, h]} 
\int_{0}^{t}(t-s)^{2 \alpha -1}   d s \\
& \leq \frac{\lambda L_{f}}{c_{1}^{2}\Gamma (2 \alpha )}
\| u_{n}-u \|_{C[0, h]} \int_{0}^{t}(t-s)^{2 \alpha -1} ds.
\end{align*}
 Then,
\begin{align}
\label{eq5}
I_{1} & \leq \frac{\lambda h^{2 \alpha}L_{f}}{c_{1}^{2}
\Gamma (2 \alpha +1 )}\| u_{n}-u \|_{C[0, h]}.
\end{align}
Concerning the second term, we have
\begin{align*}
I_{2} & \leq \frac{\lambda}{\Gamma (2 \alpha)}   
\int_{0}^{t} \frac{(t-s)^{2 \alpha -1} | f(s, u(s))|}{\left( 
\int_{0}^{t}f(x, u_{n})\, dx \right)^{2} \left( \int_{0}^{t}f(x, u)\, d x\right)^{2}} 
\left|  \left( \int_{0}^{t}f(x, u_{n})\, d x\right)^{2}
-\left( \int_{0}^{t}f(x, u)\, d x\right)^{2} \right | d s\\
& \leq \frac{\lambda c_{2}}{(c_{1}t)^{4}\Gamma (2 \alpha )}   
\int_{0}^{t}(t-s)^{2 \alpha -1} \left|  \left( \int_{0}^{t}f(x, u_{n})\, dx \right)^{2}
-\left( \int_{0}^{t}f(x, u)\, d x\right)^{2} \right | ds\\
& \leq \frac{\lambda c_{2}}{(c_{1}t)^{4}\Gamma (2 \alpha )}   
\int_{0}^{t}(t-s)^{2 \alpha -1}  \Biggl|\left( \int_{0}^{t}(f(x, u_{n})-f(x, u))\, dx \right)
\left( \int_{0}^{t}(f(x, u_{n})+f(x, u))\, d x\right) \Biggr| d s \\
& \leq \frac{2 \lambda c_{2}^{2}t}{(c_{1}t)^{4}\Gamma (2 \alpha )}   
\int_{0}^{t}(t-s)^{2 \alpha -1} \left( \int_{0}^{t}
\Biggl|f(x, u_{n})-f(x, u)\Biggr|\, dx \right ) d s \\
& \leq \frac{2 \lambda c_{2}^{2}t^{3} L_{f}}{(c_{1}t)^{4}\Gamma (2 \alpha )}   
\int_{0}^{t}(t-s)^{2 \alpha -1} \left( \int_{0}^{t}\Biggl|u_{n}(x)-u(x)\Biggr|\, dx \right ) ds \\
& \leq \frac{2 \lambda c_{2}^{2}t^{4} L_{f}}{(c_{1}t)^{4}\Gamma (2 \alpha )} 
\| u_{n}-u \|_{C[0, h]}  \int_{0}^{t}(t-s)^{2 \alpha -1}  ds \\
\end{align*}
\begin{align*}
& \leq \frac{2 \lambda c_{2}^{2} L_{f}}{c_{1}^{4}\Gamma (2 \alpha )} 
\| u_{n}-u \|_{C[0, h]}  \int_{0}^{t}(t-s)^{2 \alpha -1}  ds \\
& \leq \frac{2 \lambda c_{2}^{2} h^{2\alpha} L_{f}}{c_{1}^{4}
\Gamma (2 \alpha+1 )} \| u_{n}-u \|_{C[0, h]}.
\end{align*}
It follows that
\begin{align}
\label{eq6}
I_{2} & \leq \frac{2 \lambda c_{2}^{2} h^{2\alpha} 
L_{f}}{c_{1}^{4}\Gamma (2 \alpha+1 )} \| u_{n}-u \|_{C[0, h]}.
\end{align}
Collecting  inequalities \eqref{eq5} and \eqref{eq6} together, 
and injecting into \eqref{eq4}, we have
\begin{align*}
|A u_{n}(t) - Au(t)|& \leq  I_{1}+I_{2}\\
& \leq  \left (\frac{\lambda h^{2 \alpha}L_{f}}{c_{1}^{2}\Gamma(2 \alpha +1)}  
+  \frac{2 \lambda c_{2}^{2} h^{2\alpha} L_{f}}{c_{1}^{4}\Gamma(2 \alpha+1 )}\right)   
\| u_{n}-u \|_{C[0, h]}.
\end{align*}
Therefore,
\begin{align}
\label{eq7}
\|A u_{n} - Au\|_{C[0, h]} 
& \leq \left (\frac{\lambda h^{2 \alpha}L_{f}}{c_{1}^{2}\Gamma(2 \alpha +1 )}  
+ \frac{2 \lambda c_{2}^{2} h^{2\alpha} L_{f}}{c_{1}^{4}\Gamma(2 \alpha+1 )}\right)   
\| u_{n}-u \|_{C[0, h]}.
\end{align}
Consequently, $\| Au_n-Au\| _{C[0,h]}\to 0$ as
$n\to +\infty $, which proves that $A$ is continuous.
This ends the proof of Lemma~\ref{lemma:in1}.
\end{proof}
To finish the proof of Theorem~\ref{thm3.1}, it remains to show that
\begin{lemma}
\label{lemma:in2}
The operator $AD_{h}$ is continuous.
\end{lemma}
\begin{proof}
Let $u\in D_{h}$ and $0\leq t_1\leq t_2\leq h$. Then,
\begin{equation}
\label{8}
\begin{aligned}
|(Au)(t_{2})&- (Au)(t_{1})|\\ 
& \leq \frac{\lambda}{\Gamma (2 \alpha )} \left| 
\int_{0}^{t_{2}} (t_{2}-s)^{2 \alpha -1}   
\frac{ f(s, u(s))}{\left( \int_{0}^{t_{2}} f(x, u)\, d x\right)^{2}} d s
- \int_{0}^{t_{1}} (t_{1}-s)^{2 \alpha -1}   \frac{ f(s, u(s))}{\left( 
\int_{0}^{t_{1}}f(x, u)\, dx\right)^{2}} ds \right|\\
& \leq \frac{\lambda}{\Gamma (2 \alpha )} \left| 
\int_{0}^{t_{1}} \left ((t_{1}-s)^{2 \alpha -1} - (t_{2}-s)^{2 \alpha -1} \right) 
\frac{ f(s, u(s))}{\left( \int_{0}^{t_{1}}f(x, u)\, d x\right)^{2}} d s\right. \\
&\left.\quad + \int_{0}^{t_{1}} (t_{2}-s)^{2 \alpha -1}   
\frac{ f(s, u(s))}{\left( \int_{0}^{t_{1}}f(x, u)\, dx\right)^{2}} ds 
- \int_{0}^{t_{2}} (t_{2}-s)^{2 \alpha -1}   
\frac{ f(s, u(s))}{\left( \int_{0}^{t_{2}}f(x, u)\, d x\right)^{2}} d s \right|\\
& \leq \frac{\lambda}{\Gamma (2 \alpha )} \left| 
\int_{0}^{t_{1}} \left ((t_{1}-s)^{2 \alpha -1} - (t_{2}-s)^{2 \alpha -1} \right) 
\frac{ f(s, u(s))}{\left( \int_{0}^{t_{1}}f(x, u)\, d x\right)^{2}} d s\right. \\
& \left.\quad + \int_{0}^{t_{1}} (t_{2}-s)^{2 \alpha -1}  f(s, u(s)) \left ( 
\frac{1}{\left( \int_{0}^{t_{1}}f(x, u)\, dx\right)^{2}} 
- \frac{1}{\left( \int_{0}^{t_{2}}f(x, u)\, dx\right)^{2}} \right ) ds\right.\\
&\left.\quad - \int_{t_{1}}^{t_{2}} (t_{2}-s)^{2 \alpha -1}\frac{ f(s, u(s))}{\left( 
\int_{0}^{t_{2}}f(x, u)\, d x\right)^{2}} d s \right|\\
& \leq I_{3} + I_{4} + I_{5},
\end{aligned} 
\end{equation}
where we have, by direct calculations, that
\begin{equation}
\label{89}
\begin{aligned}
I_{3} 
& \leq \frac{\lambda c_{2}}{(c_{1}t_{1})^{2}\Gamma (2 \alpha )} \Biggl| 
\int_{0}^{t_{1}} \left ((t_{1}-s)^{2 \alpha -1} 
- (t_{2}-s)^{2 \alpha -1} \right) ds \Biggr|\\
& \leq \frac{\lambda c_{2}}{(c_{1}t_{1})^{2}\Gamma (2 \alpha +1 )} 
\Biggl| t_{1}^{2 \alpha} - t_{2}^{2 \alpha} + (t_{2}-t_{1})^{2 \alpha} \Biggr|,
\end{aligned} 
\end{equation}
\begin{equation*}
\begin{aligned}
I_{4} 
&  \leq \frac{\lambda}{\Gamma (2 \alpha )}
\frac{c_{2}t_{2}^{2\alpha -1}}{(c_{1}t_{1})^{2} (c_{1}t_{2})^{2}} 
\int_{0}^{t_{1}} \Biggl| \left( \int_{0}^{t_{2}}f(x, u)\, dx \right)^{2} 
-\left( \int_{0}^{t_{1}}f(x, u)\, d x\right)^{2}\Biggr| ds \\
& \leq \frac{\lambda}{\Gamma (2 \alpha )}
\frac{c_{2}t_{2}^{2\alpha -1}}{(c_{1}t_{1})^{2} (c_{1}t_{2})^{2}} \\
&\quad \times \int_{0}^{t_{1}} \Biggl| \left( 
\int_{0}^{t_{2}}f(x, u)\, dx + \int_{0}^{t_{1}}f(x, u)\, dx \right) 
\left( \int_{0}^{t_{2}}f(x, u)\, dx - \int_{0}^{t_{1}}f(x, u)\, dx \right) \Biggr| ds\\
& \leq \frac{\lambda}{\Gamma (2 \alpha )}\frac{c_{2}^{2}(t_{1}
+t_{2})t_{2}^{2\alpha -1}}{(c_{1}t_{1})^{2} (c_{1}t_{2})^{2}}
\Biggl| \int_{t_{1}}^{t_{2}} f(s, u(s)) ds \Biggr| \\
& \leq \frac{\lambda}{\Gamma (2 \alpha )}\frac{c_{2}^{3}(t_{1}
+t_{2})t_{2}^{2\alpha -1}}{(c_{1}t_{1})^{2} (c_{1}t_{2})^{2}} 
\left| t_{2}-t_{1} \right|
\end{aligned} 
\end{equation*}
and
\begin{equation}
\label{811}
I_{5} 
\leq \frac{\lambda c_{2}}{(c_{1}t_{2})^{2}\Gamma (2 \alpha )} 
\int_{t_{1}}^{t_{2}}  (t_{2}-s)^{2 \alpha -1} ds 
\leq \frac{\lambda c_{2}}{(c_{1}t_{2})^{2}\Gamma (2 \alpha +1)} 
\left(t_{2}-t_{1}\right)^{2\alpha}.
\end{equation}
The right hand side of inequalities \eqref{89} and \eqref{811} do not depend
on $u$ and converge to zero as $t_{2} \rightarrow t_{1}$. 
Then $\{(Au)(t):u\in D_{h}\}$ is equicontinuous 
and Lemma~\ref{lemma:in2} is proved.
\end{proof}
Taking into account that $AD_{h}\subset D_{h}$, we infer that 
$AD_{h}$ is precompact. This implies that $A$ is completely continuous. 
As a consequence of Schauder's fixed point theorem and Lemma~\ref{lem2.1}, 
we conclude that problem \eqref{1} has a local solution. 
This ends the proof of Theorem~\ref{thm3.1}.
\end{proof}

% -------------------------------------------------

\section{Continuation results}
\label{section4}

Our main contribution of this section is to prove 
a continuation theorem for the fractional Caputo nonlocal thermistor problem \eqref{1}.
First, we present the definition of noncontinuable solution.

\begin{definition}[See \cite{40}] 
\label{def4.1} 
Let $u(t)$ on $(0,\beta )$ and $\tilde{u}(t)$ on $(0,\tilde{\beta})$
be both solutions of \eqref{1}. If $\beta <\tilde{\beta}$ and 
$u(t)=\tilde{u}(t)$ for $t\in (0,\beta )$, then we say
that $\tilde{u}(t)$ can be continued to $(0,\tilde{\beta})$.
A solution $u(t)$ is noncontinuable if it has no continuation.
The existing interval of the noncontinuable solution $u(t)$ 
is called the maximum existing interval of $u(t)$.
\end{definition}

\begin{theorem} 
\label{thm4.1}
Assume that conditions $(H_1)$--$(H_3)$ are satisfied. Then
$u=u(t)$, $t\in (0,\beta )$, is noncontinuable if and if only 
for some $\eta \in \left(0,\frac{\beta }{2}\right)$ and any bounded closed subset
$S\subset [\eta ,+\infty )\times\mathbb{R}$ there exists a
$t^{\ast }\in [ \eta ,\beta )$ such that $(t^{\ast},u(t^{\ast }))\notin S$.
\end{theorem}

\begin{proof}
Suppose that there exists a compact subset
$S\subset [ \eta,+\infty )\times\mathbb{R}$ such that
$$
\left\{ (t,u(t)):t\in [ \eta ,\beta )\right\} \subset S.
$$
The compactness of $S$ implies $\beta <+\infty $. 
The remainder of the proof is given in two lemmas.
\begin{lemma}
\label{insTh2:lem1}
The limit  $\lim_{t\to \beta ^{-}}u(t)$ exists.
\end{lemma}
\begin{proof}
 Let $t_1,t_2\in [ 2\eta ,\beta )$ such that $t_1<t_2$.
 From \eqref{8}, we have
 $$
 I_{3}=I_{3,1}+ I_{3, 2},
$$
where
$$
I_{3,1} = \frac{\lambda}{\Gamma (2 \alpha )}  
\int_{0}^{\eta} \left ((t_{1}-s)^{2 \alpha -1} 
- (t_{2}-s)^{2 \alpha -1} \right) \frac{ f(s, u(s))}{\left( 
\int_{0}^{t_{1}}f(x, u)\, d x\right)^{2}} d s
$$
and
$$
I_{3,2} = \frac{\lambda}{\Gamma (2 \alpha )}  
\int_{\eta}^{t_{1}} \left ((t_{1}-s)^{2 \alpha -1} 
- \left(t_{2}-s\right)^{2 \alpha -1} \right) 
\frac{ f(s, u(s))}{\left( \int_{0}^{t_{1}}f(x, u)\, d x\right)^{2}} d s.
$$
Under assumptions $(H_1)$--$(H_3)$, there exists a positive constant $M_{3, 1}$ such that
\begin{align*}
I_{3,1}
& \leq M_{3, 1} \int_{0}^{\eta} \left ((t_{1}-s)^{2 \alpha -1} 
- (t_{2}-s)^{2 \alpha -1} \right)\\
& \leq M_{3, 1} \left| (t_{2}- \eta)^{2\alpha} 
- (t_{1}- \eta)^{2\alpha} + t_{1}^{2 \alpha}- t_{2}^{2 \alpha} \right|.
\end{align*}
Moreover, there exists a positive constant $M_{3, 2}$ such that
\begin{align*}
I_{3, 2}
&  \leq M_{3, 2} \int_{\eta}^{t_{1}} \left ((t_{1}-s)^{2 \alpha -1} 
- (t_{2}-s)^{2 \alpha -1} \right)\\
& \leq M_{3, 2} \left | (t_{2}- t_{1})^{2\alpha} 
+ (t_{1}- \eta)^{2\alpha} - (t_{1}- \eta)^{2\alpha}  \right|
\end{align*}
and we also have $I_{4}=I_{4,1}+ I_{4, 2}$, where
$$
I_{4,1} =  \int_{0}^{\eta} (t_{2}-s)^{2 \alpha -1}  f(s, u(s)) 
\left( \frac{1}{\left( \int_{0}^{t_{1}}f(x, u)\, dx\right)^{2}} 
- \frac{1}{\left( \int_{0}^{t_{2}}f(x, u)\, dx\right)^{2}} \right) ds
$$
and
$$
I_{4,2} = \int_{\eta}^{t_{1}} (t_{2}-s)^{2 \alpha -1}  f(s, u(s)) 
\left(\frac{1}{\left( \int_{0}^{t_{1}}f(x, u)\, dx\right)^{2}} 
- \frac{1}{\left( \int_{0}^{t_{2}}f(x, u)\, dx\right)^{2}} \right) ds.
$$
In the same manner as in the proof of Lemma~\ref{lemma:in2}, 
there exists positive constants $M_{4, 1}$ and $M_{4, 2}$ such that
\begin{align*}
I_{4,1} & \leq M_{4,1} |t_{2}- t_{1}|, \\
I_{4,2} & \leq M_{4,2} |t_{2}- t_{1}|.
\end{align*}
We have already proved in \eqref{811}, for some positive constant $M_{5}$, that
\begin{align*}
I_{5} & \leq M_{5} |t_{2}- t_{1}|^{2\alpha}.
\end{align*}
Therefore, we conclude that all $I_{i}$, $i=3, 4, 5$, converge to zero 
when $t_{2} \rightarrow t_{1}$. Thus, from Cauchy's convergence criterion, 
it yields that $\lim_{t\to \beta ^{-}}u(t)=u^{\ast }$.
This finishes the proof of Lemma~\ref{insTh2:lem1}. 
\end{proof}
The second step of the proof of Theorem~\ref{thm4.1} 
consists to show the following result.
\begin{lemma}
\label{insTh2:lem2}
Function $u(t)$ is continuable.
\end{lemma}
\begin{proof}
As $S$ is a closed subset, we can say that $(\beta ,u^{\ast })\in S$.
Define $u(\beta )=u^{\ast }$. Hence, $u(t)\in C[0,\beta ]$. 
Then we define the operator $K$ by
\begin{equation*}
(Kv)(t)=u_1+\frac{\lambda}{\Gamma (2 \alpha )}\int_{\beta }^{t}(t-s)^{2\alpha
-1} \frac{ f(s, v(s))}{\left( \int_{0}^{t}f(x, v)\, dx\right)^{2}} ds,
\end{equation*}
where
\begin{equation*}
u_1=u_0+\frac{\lambda}{\Gamma (2\alpha )}\int_0^{\beta }(t-s)^{2\alpha-1} 
\frac{ f(s, v(s))}{\left( \int_{0}^{t}f(x, v)\, d x\right)^{2}} ds, \quad 
v \in C([\beta ,\beta +1]),\quad  t\in [ \beta ,\beta+1].
\end{equation*}
Set
\begin{equation*}
E_{b}=\left\{ (t,v):\beta \leq t\leq \beta +1,| v| 
\leq \max_{\beta \leq t\leq \beta +1}| u_1(t)| +b\right\}
\end{equation*}
and
\begin{equation*}
E_{h}=\left\{ v\in C[\beta ,\beta +1]:\max_{t\in [ \beta ,\beta+h]}| v(t)-u_1(t)| 
\leq b,v(\beta )=u_1(\beta)\right\},
\end{equation*}
where $h=\min\left\{ \left (b \left(\frac{ \lambda M}{\Gamma(2 \alpha +1 ) 
c_{1}^{2}} \right )^{-1}\right)^{\frac{1}{2 \alpha}},1\right\}$. 
Analogously to the proof of Theorem~\ref{thm3.1}, we prove that $K$ 
is completely continuous on $E_{b}$. Indeed, let $\{v_n\}\subseteq C[\beta ,\beta +h]$.
Then $\| v_n-v\| _{[\beta ,\beta +h]}\to 0$ as $n\to +\infty $ and  
similar arguments to the one above for \eqref{eq7}, allow us to declare 
that there exists a positive constant $c_{h}$ depending on $h$ such that
\begin{equation*}
\|Kv_{n} - Kv\|_{C[\beta, \beta + h]} 
\leq  c_{h} \| v_{n}-v \|_{C[\beta, \beta +h]}.
\end{equation*}
Hence,
$\left\| (Kv_n)(t)-(Kv)(t)\right\|_{_{[\beta ,\beta+h]}}\to 0$ as $n\to +\infty$,
which yields that operator $K$ is continuous. We show that $KE_{h}$ is equicontinuous. 
For all $v\in E_{h}$, we have $(Kv)(\beta )=u_1(\beta )$ and, in view of the choice of $h$, 
it follows from hypotheses $(H_1)$ and $(H_2)$ that
\begin{align*}
\left|(Kv)(t)-u_1\right|
&= \left| \frac{\lambda}{\Gamma (2\alpha )}
\int_{\beta }^{t}(t-s)^{2\alpha -1}
\frac{ f(s, v(s))}{\left( \int_{0}^{t}f(x, v)\, dx\right)^{2}} ds\right| \\
&\leq  \frac{\lambda M}{\Gamma (2\alpha +1) c_{1}^{2}} h^{2\alpha}\leq b.
\end{align*}
Therefore, we get that $KE_{h}\subset E_{h}$. Furthermore, for any $v\in E_{h}$ 
and $\beta \leq t_1\leq t_2\leq \beta +h$, we have
$$
(Kv)(t_1)-(Kv)(t_2) = I_{6}+ I_{7}.
$$
By an analogous early calculation, there exists a positive constant $M_{6}$ such that
\begin{equation}
\label{13}
\begin{aligned}
I_{6}
& = \frac{\lambda}{\Gamma(2\alpha)} 
\int_{0}^{\beta} (t_{2}-s)^{2\alpha -1} 
\frac{ f(s, v(s))}{\left( \int_{0}^{t_{2}}f(x, v)\, d x\right)^{2}} ds
-\frac{\lambda}{\Gamma(2\alpha)} \int_{0}^{\beta} (t_{1}-s)^{2\alpha -1} 
\frac{ f(s, v(s))}{\left( \int_{0}^{t_{1}}f(x, v)\, d x\right)^{2}} ds\\
& \leq \frac{\lambda}{\Gamma(2\alpha)} \left| 
\int_{0}^{\beta} \left ( (t_{1}-s)^{2\alpha -1} - (t_{2}-s)^{2\alpha -1} \right)  
\frac{ f(s, v(s))}{\left( \int_{0}^{t_{1}}f(x, v)\, d x\right)^{2}} ds\right. \\
& \left. \quad + \frac{\lambda}{\Gamma(2\alpha)} \int_{0}^{\beta}  (t_{2}-s)^{2\alpha -1}
\left( \frac{ f(s, v(s))}{\left( \int_{0}^{t_{1}}f(x, v)\, d x\right)^{2}} ds 
-  \frac{ f(s, v(s))}{\left( \int_{0}^{t_{2}}f(x, v)\, d x\right)^{2}} ds \right) \right|\\
& \leq  M_{6} \left\{ \left |(t_{2} - \beta)^{2\alpha} 
- (t_{1} - \beta)^{2\alpha}+ t_{1}^{2\alpha} - t_{2}^{2\alpha}\right| 
+  |t_{2}-t_{1}| \right\}.
\end{aligned}
\end{equation}
An analogous treatment as in \eqref{8}--\eqref{811} yields the existence 
of a positive constant $M_{7}$ such that
\begin{equation}
\label{14}
\begin{aligned}
I_{7}
& = \frac{\lambda}{\Gamma(2\alpha)} \int_{\beta}^{t_{2}} (t_{2}-s)^{2\alpha -1} 
\frac{ f(s, v(s))}{\left( \int_{0}^{t_{2}}f(x, v)\, d x\right)^{2}} ds
-\frac{\lambda}{\Gamma(2\alpha)} \int_{\beta}^{t_{1}} (t_{1}-s)^{2\alpha -1} 
\frac{ f(s, v(s))}{\left( \int_{0}^{t_{1}}f(x, v)\, d x\right)^{2}} ds\\
& \leq M _{7}\left \{ \left | (t_{1}-\beta)^{2\alpha} 
-(t_{2}-\beta)^{2\alpha}\right|+ |t_{2}-t_{1}| + 2|t_{2} - t_{1}|^{2\alpha}\right\}.
\end{aligned} 
\end{equation}
Since the right side of inequalities \eqref{13} and \eqref{14} go to zero 
as $t_{2} \rightarrow t_{1}$, we deduce that $\{ (Kv)(t):v\in E_{h}\}$ is
equicontinuous. Consequently, $K$ is completely continuous. Then, Schauder's fixed
point theorem can be applied to obtain that operator $K$ has a fixed point 
$\tilde{u}(t)\in E_{h}$. On other words, we have
\begin{equation*} 
\begin{aligned}
\tilde{u}(t)
&= u_1+\frac{\lambda}{\Gamma (2\alpha )}\int_{\beta}^{t}
(t-s)^{2\alpha -1} \frac{ f(s, \tilde{u}(s))}{\left( 
\int_{0}^{t}f(x, \tilde{u}(x))\, dx\right)^{2}} ds,\\
&= u_0+\frac{\lambda}{\Gamma (2\alpha )}
\int_0^{t}(t-s)^{2\alpha -1} \frac{ f(s, \tilde{u}(s))}{\left( 
\int_{0}^{t}f(x, \tilde{u}(x))\, dx\right)^{2}} ds,\quad t\in [ \beta ,\beta +h],
\end{aligned}
\end{equation*}
where
\[
\tilde{u}(t)=\begin{cases}
u(t), & t\in (0,\beta ] \\
\tilde{u}(t), & t\in [ \beta ,\beta +h].
\end{cases}
\]
It follows that $\tilde{u}(t)\in C[0,\beta +h]$ and
\begin{equation*}
\tilde{u}(t)=u_0+\frac{\lambda}{\Gamma (2\alpha )}
\int_0^{t}(t-s)^{2\alpha-1}\frac{ f(s, \tilde{u}(s))}{\left( 
\int_{0}^{t}f(x, \tilde{u}(x))\, dx\right)^{2}} ds.  
\end{equation*}
Therefore, according to Lemma~\ref{lem2.1}, $\tilde{u}(t)$ is a solution of
\eqref{1} on $(0,\beta +h]$. This is absurd because $u(t)$ is
noncontinuable. This completes the proof of Lemma~\ref{insTh2:lem2}.
\end{proof}
Theorem~\ref{thm4.1} follows from Lemmas~\ref{insTh2:lem1}
and \ref{insTh2:lem2}.
\end{proof}

\begin{remark} 
\label{rmk4.1}
Uniqueness of solution to problem \eqref{1}  
is easily derived from the proof of Theorem~\ref{thm4.1}
for a well chosen $\lambda$.
\end{remark}

% -------------------------------------------------

\section{Global existence}
\label{section5}

Now we provide two sets of sufficient conditions
for the existence of a global solution for \eqref{1}
(Theorems~\ref{thm5.2} and \ref{thm5.2b}). 
We begin with an auxiliary lemma. 

\begin{lemma} 
\label{thm5.1}
Suppose that conditions $(H_1)$--$(H_3)$ hold. Let $u(t)$ be a solution
of \eqref{1} on $(0,\beta )$. If $u(t)$ is bounded on $[\tau ,\beta )$ 
for some $\tau >0$, then $\beta =+\infty$.
\end{lemma}

\begin{proof}
Follows immediately from the results of Section~\ref{section4}.
\end{proof}

\begin{theorem} 
\label{thm5.2}
Suppose that conditions $(H_1)$--$(H_3)$ hold.
Then \eqref{1} has a solution in $C([0,+\infty ))$.
\end{theorem}

\begin{proof}
The existence of a local solution $u(t)$ of \eqref{1} is ensured thanks to
Theorem~\ref{thm3.1}. We already know, by Lemma~\ref{lem2.1}, 
that $u(t)$ is a also a solution to the integral equation
\begin{equation*}
u(t)=u_0+\frac{\lambda}{\Gamma (2\alpha )}\int_0^{t}(t-s)^{2\alpha
-1}\frac{ f(s, u(s))}{\left( \int_{0}^{t}f(x, u(x))\, dx\right)^{2}}ds.
\end{equation*}
Suppose that the existing interval of the noncontinuable solution $u(t)$ is
$(0,\beta)$, $\beta<+\infty$. Then,
\begin{align*}
|u(t)|
&= \left| u_0+\frac{\lambda}{\Gamma (2\alpha)}
\int_0^{t}(t-s)^{2\alpha -1}\frac{ f(s, u(s))}{\left( 
\int_{0}^{t}f(x, u(x))\, dx\right)^{2}}ds\right| \\
& \leq |u_0|+ \frac{\lambda}{\Gamma (2\alpha )} \frac{1}{(c_{1}t)^{2}} 
\int_0^{t}(t-s)^{2\alpha -1}\left| f(s, u(s))\right|ds\\
&\leq |u_0| + \frac{\lambda }{\Gamma (2\alpha )} 
\frac{1}{c_{1}^{2}} \int_0^{t} \frac{|u(s)|}{(t-s)^{1- 2\alpha}} ds.
\end{align*}
By Lemma~\ref{lem2.1}, there exists a constant $k(\alpha)$ such that, 
for $t\in (0, \beta)$, we have
\begin{align*}
| u(t)| & \leq |u_0| + k |u_0|\frac{\lambda }{\Gamma (2\alpha )} 
\frac{1}{c_{1}^{2}} \int_0^{t} (t-s)^{2\alpha-1} ds,
\end{align*}
which is bounded on $(0, \beta)$. Thus, by Lemma~\ref{thm5.1}, 
problem \eqref{1} has a solution $u(t)$ on $(0,+\infty )$.
\end{proof}

Next we give another sufficient condition ensuring 
global existence for \eqref{1}.

\begin{theorem} 
\label{thm5.2b}
Suppose that there exist positive constants
$c_{3}$, $c_{4}$ and  $c_{5}$ such that
$c_{3} \leq |f(s, x)|\leq c_{4}|x|+ c_{5}$.
Then \eqref{1} has a solution in $C([0,+\infty ))$.
\end{theorem}

\begin{proof}
Suppose that the maximum existing interval of $u(t)$ is
$(0,\beta )$, $\beta<+\infty$. We claim that $u(t)$ 
is bounded on $[\tau ,\beta )$ for any $\tau \in (0,\beta )$. 
Indeed, we have
\begin{align*}
|u(t)|
&= \left| u_0+\frac{\lambda}{\Gamma (2\alpha )}
\int_0^{t}(t-s)^{2\alpha -1}\frac{ f(s, u(s))}{\left( 
\int_{0}^{t}f(x, u(x))\, dx\right)^{2}}ds\right| \\
& \leq |u_0|+ \frac{\lambda}{\Gamma (2\alpha )} 
\frac{1}{(c_{1}\tau)^{2}} \int_0^{t}(t-s)^{2\alpha -1}\left| f(s, u(s))\right|ds\\
&\leq |u_0|+\frac{\lambda}{\Gamma (2\alpha )} 
\frac{c_{3}}{(c_{1}\tau)^{2}}\int_0^{t}(t-s)^{2\alpha -1} ds 
+ \frac{\lambda }{\Gamma (2\alpha )} \frac{c_{2}}{(c_{1}\tau)^{2}}
\int_0^{t}  \frac{|u(s)|}{(t-s)^{1- 2\alpha}} ds.
\end{align*}
If we take 
$$
w(t)=|u_0|+ \frac{\lambda}{\Gamma (2\alpha )} 
\frac{c_{3}}{(c_{1}\tau)^{2}}\int_0^{t}(t-s)^{2\alpha -1}ds,
$$ 
which is bounded, and
$$
a=\frac{\lambda c_{2}}{\Gamma (2\alpha )} \frac{1}{(c_{1}\beta)^{2}},
$$
it follows,  according with Lemma~\ref{lem5.1}, that
$v(t)=|u(t)|$ is bounded. Thus, by Lemma~\ref{thm5.1}, 
\eqref{1} has a solution $u(t)$ on $(0,+\infty )$.
\end{proof}

% -------------------------------------------------

\section{Conclusion}
\label{section6}

In our paper we consider a prototype of electrical 
conductivity that depends strongly in both time 
and temperature. The model relates to modern developments 
of thermistors, where fractional PDEs have a crucial role.
It turns out that available computational methods are not 
theoretically sound in the sense they rely on results 
of local existence. The main novelty of our paper is that
we prove global existence for a nonlocal 
thermistor problem with fractional differentiation 
in the Caputo sense. Moreover, we extend some 
results of continuation and global existence to 
the fractional order initial value thermistor problem.
The proofs rely on Schauder's fixed point theorem.
We trust our results will have a positive impact 
on the development of computer mathematics
to fractional thermistor devices.

% -------------------------------------------------

\section*{Acknowledgments}

The authors were supported by the \emph{Center for Research
and Development in Mathematics and Applications} (CIDMA)
of University of Aveiro, through Funda\c{c}\~ao
para a Ci\^encia e a Tecnologia (FCT),
within project UID/MAT/04106/2013.
They are grateful to two anonymous referees, 
for several comments and suggestions of improvement.

% -------------------------------------------------

\section*{Competing interests} 

The authors declare that they have no competing interests.

\small

% -------------------------------------------------

% -------------------------------------------------

\end{document}